\documentclass{amsart}
\usepackage{amsmath,amssymb}
\usepackage[all]{xy}

\theoremstyle{plain}
\newtheorem{introtheorem}{Theorem}
\newtheorem*{introcorollary}{Corollary}
\newtheorem{theorem}{Theorem}[section]

\newtheorem{lemma}[theorem]{Lemma}
\newtheorem{corollary}[theorem]{Corollary}

\newtheorem*{proposition*}{Proposition}

\theoremstyle{definition}

\newtheorem{example}[theorem]{Example}

\theoremstyle{remark}
\newtheorem{remark}[theorem]{Remark}

\newcommand{\secref}[1]{Section~\ref{#1}}
\newcommand{\thmref}[1]{Theorem~\ref{#1}}

\newcommand{\lemref}[1]{Lemma~\ref{#1}}
\newcommand{\corref}[1]{Corollary~\ref{#1}}

\def\Q{{\mathbb Q}}

\def\map{\mathrm{map}}

\def\cat0{\mathrm{cat}_0}

\def\aut{\mathrm{aut}_1\,}
\def\ev{\mathrm{ev}}

\begin{document}

\title[Gottlieb Groups of Function Spaces]
{Gottlieb Groups of Function Spaces}

\author{Gregory  Lupton}

\address{Department of Mathematics,
           Cleveland State University,
           Cleveland OH 44115}

\email{G.Lupton@csuohio.edu}

\author{Samuel Bruce Smith}

\address{Department of Mathematics,
   Saint Joseph's University,
   Philadelphia, PA 19131}

\email{smith@sju.edu}

\date{\today}



\begin{abstract}
We  analyze the Gottlieb groups of  function spaces.   Our results lead to  explicit  decompositions of the Gottlieb groups of many function spaces $\map(X,Y)$---including the (iterated) free loop space of $Y$---directly in terms of the Gottlieb groups of $Y$.  More generally, we give explicit decompositions of the generalized Gottlieb groups of  $\map(X,Y)$ directly in terms of generalized Gottlieb groups of $Y$.   Particular cases of our results relate to the torus homotopy groups of Fox.  We draw some consequences for the classification of $T$-spaces and $G$-spaces. For $X$, $Y$ finite and $Y$ simply connected, we give a formula for the ranks of the Gottlieb groups of $\map(X,Y)$ in terms of the Betti numbers of $X$ and the ranks of the Gottlieb groups of $Y$.  Under these hypotheses,  the Gottlieb groups of $\map(X,Y)$ are finite groups in all but finitely many degrees.
\end{abstract}

\thanks{This work was partially supported by a grant from the Simons Foundation (\#209575 to Gregory Lupton).}

\maketitle

\section{Introduction: Description of Results}%
\label{sec:intro}

Let  $\aut  Y$ denote the component of the function space of (unbased) self-homotopy equivalences of $Y$ that consists of self maps (freely) homotopic to the identity map of $Y$.  For $Y$  a based space, evaluation at the basepoint of $Y$ gives the
\emph{evaluation map} $\omega \colon \aut Y\rightarrow Y$, which induces a homomorphism on $n$th homotopy groups
$$\omega_{\#}\colon \pi_n\big(\aut Y \big) \to \pi_n(Y).$$
The \emph{$n$th Gottlieb group of $Y$}, denoted $G_{n}(Y)$, 
is the image of $\omega_{\#}$ in $\pi_n(Y)$ \cite{Go1}.

Because $\omega\colon \aut Y \rightarrow Y$ may be identified
with the connecting map of the universal fibration for fibrations
with fibre $Y$, the Gottlieb groups are important universal
objects that feature in  a variety of contexts.    
Recently, for example, they have appeared as objects of interest in results in symplectic topology \cite[Lem.2.2]{McD08}, and in string topology \cite[Th.2]{F-T-V04}.  
Unfortunately, it has proved difficult to calculate
Gottlieb groups.  In part, this may be due to a lack of functoriality:  a map $f\colon
X \to Y$ in general does not satisfy $f_\#(G_n(X)) \subseteq G_n(Y)$.

Our goal in this paper is to analyze the Gottlieb groups of function spaces, including important cases such as the free loop space.  The prototype of our results is the following.  Let $\Lambda Y = \map(S^1, Y;0)$ denote the (null component of the) free loop space of $Y$. 

\begin{introtheorem}\label{introthm: Gottlieb free loop}
For $n \geq 1$, we have 
 $$G_n(\Lambda  Y) \cong  G_{n}(Y) \oplus G_{n+1}(Y).$$
\end{introtheorem}

Note that Gottlieb groups are abelian; the sum in the above is the direct sum of abelian groups.
This basic result may be generalized greatly.  We indicate the various generalizations in this introduction; the actual theorems in the body of the paper are more general than the samples we now give.

In fact we are able to give a similar decomposition for the (generalized) Gottlieb groups of a general function space.   
Let $\map(X,Y)$ denote the component of the null map in the function space of (unbased) maps from $X$ to $Y$.  We will often write this as $\map(X,Y;0)$, to emphasize that we are dealing only with null components here.  In what follows, our hypotheses on the spaces $X$ and $Y$ are very mild: they are specified   at the start of \secref{sec: Gottlieb function space}.  We assume $X$ and $Y$ are connected, but do not require any higher connectivity or finiteness hypotheses.     
For a suspension $\Sigma A$, denote by $[\Sigma A, Z]$ the group of based homotopy classes of based maps $\Sigma A \to Z$.  Then post-composition with the evaluation map $\omega\colon \aut Z \to Z$ induces a homomorphism of groups of homotopy classes
$$\omega_*\colon [\Sigma A, \aut Z] \to [\Sigma A, Z],$$
whose image we denote by $\mathcal{G}(\Sigma A, Y)$, and call a \emph{generalized Gottlieb group} (see \cite{Var69}).    Because $\aut Z$ is an H-space, the group $[\Sigma A, \aut Z]$ is abelian.  Therefore, this latter description makes clear that $\mathcal{G}(\Sigma A, Z)$ is an abelian subgroup of $[\Sigma A, Z]$.   For the space $A$, which always appears suspended, we allow it to be disconnected, so that, for example, we may obtain a (bouquet of) circle(s) in $\Sigma A$.  Notice that, by taking $A = S^{n-1}$ for $n \geq 1$, we obtain $\mathcal{G}(\Sigma S^{n-1}, Z) = G_n(Z)$, the ordinary Gottlieb group.

\begin{introtheorem}[\corref{cor: Gottlieb function space}]%
\label{thm: intro Gottlieb function space} 
For any $A$, we have an isomorphism of abelian groups
$$\mathcal{G}\big(\Sigma A , \map(X, Y;0)\big) \cong \mathcal{G}(\Sigma A, Y) \oplus \mathcal{G}(\Sigma(A \wedge X), Y).$$
In particular, for each $n \geq 1$, we have an isomorphism of abelian groups
$$G_n\big(\map(X, Y;0)\big) \cong G_n(Y) \oplus \mathcal{G}(\Sigma^{n} X, Y).$$
\end{introtheorem}

\thmref{introthm: Gottlieb free loop} follows from this result by setting $X = S^1$.  Another interesting special case is given by setting $X$ a wedge of spheres, which yields the following identity (included in \thmref{thm: General Gottlieb group}):

\begin{introcorollary}
For $X = S^{i_1} \vee \cdots \vee S^{i_k}$, and for $n \geq 1$, we have an isomorphism
$$G_n\big(\map(X,Y) \big) \cong G_n(Y) \oplus \bigoplus_{r = 1, \dots, k} G_{n+i_r}(Y).$$
\end{introcorollary}
\noindent{}For $X$ a single sphere, of course, we have  $G_n\big(\map(S^p, Y;0)\big) \cong G_n(Y) \oplus G_{n+p}(Y)$.

For $X$ a product of spaces, \thmref{thm: intro Gottlieb function space} may be applied directly.  However, in this case, it is also possible to use the result iteratively, so as to decompose the right-hand sides of the identity into simpler pieces.   This approach leads, in particular,  to an explicit decomposition of the Gottlieb groups of function spaces of the form
$$\map\big(  (S^{a_1} \vee \cdots \vee S^{a_k}) \times \overset{N\textrm{-}\mathrm{times}}{\cdots} \times 
(S^{b_1} \vee \cdots \vee S^{b_l}),   Y ; 0 \big)$$
 directly in terms of the Gottlieb groups of $Y$.  Also, in these cases, some interesting combinatorial expressions arise.  We give one such decomposition here, which is included in \corref{cor: Gottlieb iterated bouquet}.  \secref{sec:Free Bouquets} gives more general results, along similar lines.
 
\begin{introtheorem}%
\label{thm: intro Gottlieb iterated bouquet} 
For $N \geq 1$,  let $\Lambda ^N(Y)$ denote the iterated free loop space of $Y$ (see \secref{sec:Free Bouquets} for a definition).  Then, for each $n \geq 1$,  we have%
$$G_n\big(\Lambda ^N(Y)\big) \cong \bigoplus_{j = 0}^{N}  {N\choose j} G_{n+j}(Y).$$
\end{introtheorem}
\noindent{}In these formulas, a notation of the form $k G$, for $k \geq 1$ an integer and $G$ an abelian group, denotes the direct sum of $k$ copies of $G$.  These decompositions are quite remarkable, given the apparent difficulty of identifying Gottlieb groups for ``small" spaces such as manifolds, or  cell complexes with few cells, including spheres (see \cite{Go-Mu08}).   

The iterative use of \thmref{thm: intro Gottlieb function space} depends on the exponential law, whereby we we may identify
the function spaces $\map(A, \map(B, Y))$ and $\map(A \times B, Y)$.   In the case of the iterated free loop space $\Lambda ^N(Y)$, repeated use of this allows us to identify $\Lambda ^N(Y)$ with the function space $\map(T^N, Y; 0)$, where $T^N$ denotes the $N$-torus $S^1 \times \cdots \times S^1$.   So the identity of \thmref{thm: intro Gottlieb iterated bouquet} gives $G_n\big(\map(T^N,Y;0) \big)$ in terms of $G_*(Y)$.  In this case, there is a relation between our identities here, and similar identities noticed by Fox and others, concerning the so-called \emph{Fox torus homotopy groups} \cite{Fox}, and the \emph{Fox-Gottlieb groups} introduced in \cite{G-G-W}.  We discuss the connection in \secref{sec:Free Bouquets}.

In a final generalization of our basic result,  we relativize \thmref{introthm: Gottlieb free loop}.  This begins with replacing the free loop space $\Lambda  Y$ with a pullback of the free loop fibration $\Lambda  Y \to Y$ over a map $f\colon X \to Y$, which we denote $L_f Y$. In this setting, we obtain a relation that involves the so-called \emph{relative Gottlieb groups}, or the \emph{evaluation subgroups of a map}.  We review the definition of these groups now, before indicating our result. 

Let $f \colon X \rightarrow Y$ be a based map. Denote by $\map(X, Y; f)$ the path component of the
space of (unbased) maps $X\to Y$ that consists of maps (freely)
homotopic to $f$. Evaluation at the basepoint of $X$ gives the
\emph{evaluation map} $\omega \colon \map(X, Y; f)\rightarrow Y$.
The \emph{$n$th evaluation subgroup of $f$}, denoted $G_{n}(Y, X;
f)$, is the image of $\omega_{\#}$ in $\pi_n(Y)$ (cf.~\cite[p.731]{Go1}). The Gottlieb
group $G_{n}(Y)$ occurs as the special case in which $X=Y$ and $f=
1_{Y}$.

From the construction of $L_fY$, we obtain a canonical  ``whisker" map $\phi\colon \Lambda X \to L_fY$.  Our basic result in this relative setting, which may be viewed as a relative version of the prototypical  
\thmref{introthm: Gottlieb free loop}, is the following.  

\begin{introtheorem}[\thmref{thm: relative Gottlieb free loop}]%
\label{thm:intro relative Gottlieb}
For each $n \geq 2$, we have an isomorphism of abelian groups
$$G_n\big(L_fY, \Lambda  X; \phi \big) \cong G_n(X) \oplus G_{n+1}(Y,X;f).$$
\end{introtheorem}

As with \thmref{thm: intro Gottlieb function space}, it is possible to handle iterates of this relativized construction, although we do not develop this direction here.

Our results lead to a strong consequence for the global structure of the Gottlieb groups of a function space.  Let $\beta_i(X)$ denote the $i$th Betti number of $X$, and $\gamma_i(Y)$---which we propose calling the $i$th \emph{Gottlieb number} of $Y$---denote the rank (as an abelian group) of $G_i(Y)$. We deduce the following formula:

\begin{introtheorem}\label{thm: gamma formula}
Let $X$ and $Y$ be finite complexes, with $Y$ simply connected.  Then we have 
$$\gamma_n\big( \map(X,Y;0) \big) = \sum_{i=0}^{\textrm{dim}\,X}\ \beta_i(X) \gamma_{n+i}(Y),$$
for $n \geq 1$, where $\textrm{dim}\,X$ is the dimension of $X$.
\end{introtheorem}
\noindent{}It follows (\corref{cor:finite Gottlieb}) that, under these hypotheses,  $G_n\big( \map(X,Y;0) \big)$  \emph{is a finite group for all but finitely many $n$}.

The paper is organized as follows.  In \secref{sec: Gottlieb function space} we set hypotheses and show our basic results, including \thmref{thm: intro Gottlieb function space}.  In \secref{sec:Free Bouquets}  we focus on the iterated bouquet spaces and establish several explicit decompositions, such as that of \thmref{thm: intro Gottlieb iterated bouquet}.  This section also includes our discussion of the Fox torus homotopy groups.  In  \secref{sec:Pullbacks}, we carry out the relativization of \thmref{introthm: Gottlieb free loop} and obtain \thmref{thm:intro relative Gottlieb}.  In \secref{sec:String} we deduce several consequences of our results, including \thmref{thm: gamma formula} above.

\section{Generalized Gottlieb Groups of Function Spaces}\label{sec: Gottlieb function space}

We begin by establishing notation and carefully specifying our hypotheses.  
We use  $X$, $Y$ and $A$ to denote based spaces,  with $x_0$ the  basepoint of $X$.  Our hypotheses on these spaces are driven by the proofs in this section: we wish to make several identifications of---based and unbased---function spaces using the exponential law, and we want evaluation maps to be fibrations---actually, to induce Barratt-Puppe sequences of homotopy sets.   To this end, we assume that spaces $X$, $Y$, and $A$ are locally compact  and  Hausdorff.  The function space $\map(X,Y)$ has the compact-open topology.
Then $X$, $Y$, and $\map(X,Y)$ are compactly generated spaces, and the identifications we use are given in Theorems 5.6 and 5.12 of \cite{Ste67}.  Note that Hausdorff-ness of $\map(X,Y)$ is inherited from $Y$.  Under our hypotheses, there is no need to adjust the topologies on the products or function spaces that arise here, in order to apply the results of \cite{Ste67}. 
  Also,  we assume that spaces $X$, $Y$, and $A$ are well-pointed, from which it follows that the evaluation maps that we use are (Hurewicz) fibrations \cite[Th.2.8.2]{Spa81}.   We suppose that spaces $X$ and $Y$ are connected but, as mentioned in the introduction, a space $A$ (or $B$) that appears  suspended may be disconnected.  Note, in particular, that we do not require any finiteness hypotheses on our spaces $X$, $Y$, or $A$, and neither do we require that $X$ or $Y$ be simply connected, or that any of these spaces be (of the homotopy type of) a  CW complex.    For based spaces, $\Sigma A$ denotes the reduced suspension, and $A \wedge X$ the smash product. 
We also use the notation and vocabulary introduced before \thmref{thm: intro Gottlieb function space} of the Introduction.

Generally speaking, we are interested in identifying  generalized Gottlieb groups of a function space, that is, the image of a homomorphism 
$$[\Sigma A, \aut \map(X,Y;0)] \to [\Sigma A, \map(X,Y;0)]$$
induced by the evaluation map of $\map(X,Y;0)$.
To this end, we have the following general result.  We emphasize that, here, $\map(X,Y)$ denotes the null component $\map(X,Y;0)$. 

\begin{lemma}\label{lem: retractions}
Let  $\ev\colon \map(X,Y) \to Y$ be the evaluation fibration given by $\ev(g) = g(x_0)$.  Then  $\omega_Y\colon \aut Y \to Y$ is a retract (as a map) of $(\omega_Y)_*\colon \map(X, \aut Y) \to \map(X,Y)$, which in turn is a retract of $\omega_{\map(X,Y)}\colon \aut \map(X,Y) \to \map(X,Y)$.   
\end{lemma}

\begin{proof}
For consider the following commutative diagram:
$$\xymatrix{\aut Y \ar[d]_{\omega_Y} \ar[rr]_-{\sigma_{\aut Y}}  & & \map(X, \aut Y)
\ar[d]_{(\omega_Y)_*} 
\ar@/_1pc/[ll]_-{\ev_{\aut Y}}  \ar[r]_{\Phi} & \aut \map(X,Y) \ar@/_1pc/[l]_-{r}
\ar[d]^{\omega_{\map(X,Y)}}\\
Y \ar[rr]^-{\sigma_Y} &  & \map(X,Y) \ar@{=}[r] \ar@/^1pc/[ll]^-{\ev_Y}  & \map(X,Y) }$$
The vertical maps are evaluation (at the basepoint) maps, or the map induced by the evaluation map in the case of the middle one.   In the left-hand square, the retractions and their sections are the usual evaluation maps of the form 
$\ev_Y\colon \map(X,Y;0) \to Y$, and $\sigma_Y \colon Y \to \map(X,Y;0)$, defined by evaluation at the base point of $X$, and the null section, respectively.  The latter means $\sigma_Y(y) = C_y$, with $C_y(x) = y$, the null map at $y$.  Note that the null map $X \to  \aut Y$ maps each point of $X$ to the \emph{identity} of $Y$, which is  the base point in $\aut Y$.

The maps $\Phi$ and $r$ may be described as follows.    The exponential law  gives homeomorphisms
$$\map(X, \aut Y; 0) \equiv \map(X \times Y, Y; \pi_2)$$
and
$$\aut \map(X,Y;0) \equiv \map( \map(X,Y;0) \times X, Y; \mathrm{EV}).$$
Here, $\mathrm{EV}\colon \map(X,Y;0) \times X \to  Y$ denotes the ``big" evaluation map, given by $\mathrm{EV}(f,x) = f(x)$. 
For the spaces on the right-hand sides, we have an obvious section and retraction  
$$\xymatrix{\map(X \times Y, Y; \pi_2) \ar[rrr]_-{\big(T \circ (\ev_Y \times 1)\big)^*}  & & & \map( \map(X,Y;0) \times X, Y; \mathrm{EV})  \ar@/_1pc/[lll]_-{\big( (1 \times \sigma_Y) \circ  T\big)^*}  }$$
where $\xymatrix{X \times Y \ar@{<->}[r]^{T}_{\equiv} &Y \times X}$ denotes the switching map.  Via the preceding adjunctions, we may define $\Phi$ and $r$, then, as
$$\Phi(g)(f)(x) := g(x)\big( f(x_0)\big) \qquad \mathrm{and} \qquad r(\psi)(x)(b) := \psi(C_b)(x),$$
for $g \in \map(X, \aut Y)$, $f \in \map(X,Y)$, $x \in X$, and $\psi \in \aut \map(X,Y)$, $x \in X$, $b\in Y$.

A straightforward check shows that both squares commute in either direction, and that we have the desired retractions. 
\end{proof}

\begin{remark}
Collapsing the two retractions from \lemref{lem: retractions} yields a retraction (of maps)
$$\xymatrix{\aut Y \ar[d]_{\omega_Y} \ar[r]^-{\iota}  & \aut \map(X,Y) \ar[d]^{\omega_{\map(X,Y)}}  
 \ar[rr]^-{(\ev_Y)_*\circ (\sigma_Y)^*} & &\aut Y 
\ar[d]^{\omega_Y} \\
Y \ar[r]_-{\sigma_Y} &  \map(X,Y) \ar[rr]_-{\ev_Y} & & Y }$$
In place of the sectioned fibration $\ev_Y\colon \map(X,Y) \to Y$, we could consider a general sectioned fibration $p\colon E \to B$, with section $\sigma\colon B \to E$ (so $E$ dominates $B$).  In this case,  we do have a diagram
$$\xymatrix{\aut B \ar[d]_{\omega_B} \ar@{..>}[r] & \aut E \ar[d]^{\omega_{E}}  
 \ar[rr]^-{p_*\,\circ\, \sigma^*} & &\aut B 
\ar[d]^{\omega_B} \\
B \ar[r]_-{\sigma} &  E \ar[rr]_-{p} & & B .}$$
 The right-hand square here yields the standard fact that, since $p$ has a right homotopy inverse, we have $p_\#\big( G_*(E)\big) \subseteq G_*(B)$.  On  the other hand, the lack of a filler $\aut B \to \aut E$ in the left-hand square, in general, reflects the fact that $G_n(-)$ fails to be a functor.   This makes  clear that we are relying on the particular properties of the evaluation fibration $\ev_Y\colon \map(X,Y) \to Y$ in order to obtain our results.
\end{remark}

We now give our main result.

\begin{theorem}\label{thm: split short Gottlieb}
For $X$, $Y$, and $A$ satisfying the hypotheses above, there is a split short exact sequence of abelian groups
$$\xymatrix{ 0 \ar[r] & \mathcal{G}(\Sigma(A\wedge X), Y) \ar[r] &
\mathcal{G}\big(\Sigma A, \map(X,Y;0) \big) \ar[r]_-{(\ev_Y)_*} & \mathcal{G}(\Sigma A, Y) \ar@/_1pc/[l]_-{(\sigma_{Y})_*} \ar[r] & 0.}$$
\end{theorem}

\begin{proof}
Consider the following commutative diagram.
\begin{equation}\label{eq: evaluation retract}
\xymatrix{
 & & \aut \map(X,Y) \ar@<1ex>[d]^{r} \\
 \map_*(X, \aut Y) \ar[d]_{(\omega_Y)_{*}} \ar[rr]_-{j_{\aut Y}}  & & \map(X, \aut Y) \ar[r]_-{\ev_{\aut Y}} 
\ar[d]_{(\omega_Y)_{*}}  \ar@<1ex>[u]^{\Phi}
   & \aut Y  \ar[d]^{\omega_Y} \ar@/_1pc/[l]_-{\sigma_{\aut Y}} \\
\map_*(X,Y) \ar[rr]_-{j_Y}  &  & \map(X,Y)  \ar[r]^-{\ev_Y}    & Y \ar@/^1pc/[l]^-{\sigma_Y}  } 
\end{equation}
The horizontal rows are fibre sequences of evaluation fibrations (restricted to the null component), each with a section.  The vertical maps are the evaluation map and its induced maps.  The maps $\Phi$ and $r$ are those from \lemref{lem: retractions}.  

Now apply $[\Sigma A, -]$ to this diagram.  Then the lower part of the diagram yields a ladder of split extensions   
$$\xymatrix{
 [\Sigma(A \wedge X), \aut Y] \ \ \ \ar[d]_{(\omega_Y)_{*}} \ar@{>>->}[r]_-{j}  &  [\Sigma A,\map(X, \aut Y)] \ar@{->>}[r]_-{\ev_{*}} 
\ar[d]_{( \omega_{\map(X,Y)} \circ \Phi)_{*}}  
   &  [\Sigma A, \aut Y]  \ar[d]^{(\omega_Y)_*} \ar@/_1pc/[l]_-{\sigma_{*}} \\
[\Sigma(A \wedge X),Y] \ \ \ \ar@{>>->}[r]_-{j}    & [\Sigma A, \map(X,Y)]  \ar@{->>}[r]^-{(\ev)_*}    & [\Sigma A, Y] \ar@/^1pc/[l]^-{(\sigma)_*}  } 
$$
All maps are those induced from (\ref{eq: evaluation retract}), and we have omitted subscripts off the horizontal ones.  We have used the exponential law for based mapping spaces to write the left-hand terms in this way: the identification
$$[\Sigma A, \map_*(X,Z)] = \pi_0\big( \map_*(\Sigma A, \map_*(X,Z))\big) = 
\pi_0\big( \map_*(\Sigma A \wedge X,Z)\big) = [\Sigma A \wedge X,Z]$$
is natural with respect to maps induced by maps of $Z$.  Now  consider the image of the upper split extension in the lower. Generally, the image of an exact sequence in an exact sequence fails to be exact.  However, here we have the image of one split, short exact sequence in another.  One easily checks---a standard diagram chase---that the sequence of image subgroups is also split short exact.  Furthermore, the image subgroups are abelian, since $\aut Y$ is an H-space, and hence the groups in the upper sequence are abelian.  The first and last image subgroups are  $\mathcal{G}(\Sigma(A\wedge X), Y)$ and $\mathcal{G}(\Sigma A, Y)$ by definition.  The middle image subgroup is  $\mathcal{G}\big(\Sigma A, \map(X,Y;0) \big)$ because of the factorization $(\omega_Y)_* =  ( \omega_{\map(X,Y)} \circ \Phi)_{*}\colon \map(X, \aut Y) \to \map(X,Y)$ as in (\ref{eq: evaluation retract}), together with the fact that the sequences are split short exact.
\end{proof}

\begin{corollary}\label{cor: Gottlieb function space}
We have an isomorphism of abelian groups
$$\mathcal{G}\big(\Sigma A , \map(X, Y;0)\big) \cong \mathcal{G}(\Sigma A, Y) \oplus \mathcal{G}(\Sigma(A \wedge X), Y).$$
In particular, for each $n \geq 1$, we have an isomorphism of abelian groups
$$G_n\big(\map(X, Y;0)\big) \cong G_n(Y) \oplus \mathcal{G}(\Sigma^{n} X, Y). \qed$$
\end{corollary}

Now specialize to the case in which $X = S^{i_1} \vee \cdots \vee S^{i_k}$ or, more generally, in which we have $\Sigma X = S^{i_1+1} \vee \cdots \vee S^{i_k+1}$.

\begin{theorem}\label{thm: General Gottlieb group}
If $\Sigma X = S^{i_1+1} \vee \cdots \vee S^{i_k+1}$, then we have an isomorphism
$$G_n\big(\map(X,Y) \big) \cong G_n(Y) \oplus \bigoplus_{r = 1, \dots, k} G_{n+i_r}(Y),$$
for $n \geq 1$.   More generally, given a product decomposition $X = X' \times X''$, for which we have $\Sigma X' = S^{i_1+1} \vee \cdots \vee S^{i_k+1}$, then for $n \geq 1$ we have an isomorphism
$$G_n\big(\map(X,Y) \big) \cong G_n\big(\map(X'',Y) \big)\oplus \bigoplus_{r = 1, \dots, k} G_{n+i_r}\big(\map(X'',Y) \big).$$
\end{theorem}

\begin{proof}
In the first case, we have $\Sigma^n X \simeq S^{n+i_1} \vee \cdots \vee S^{n+i_k}$, and hence $[\Sigma^n X, Z] \cong
\oplus_{r = 1, \dots, k} [S^{n+i_r}, Z]$ for any $Z$.  It follows that 
$\mathcal{G}(\Sigma^n X, Y) \cong \oplus_{r = 1, \dots, k} G_{n+i_r}(Y)$, and the first assertion follows immediately from \corref{cor: Gottlieb function space}.
For the more general statement, simply write $\map(X,Y)$ as $\map(X', \map(X'', Y))$ and apply the first statement.  This contains the first statement as the case in which $X'' = *$.
\end{proof}

\section{Explicit Decompositions: Gottlieb Groups of Iterated Free Bouquet Spaces}\label{sec:Free Bouquets}

In this section, we consider the particular cases in which $X$ is a product, each factor of which splits as a wedge of spheres after one suspension.  Here,  the first isomorphism of \thmref{thm: General Gottlieb group} may be applied recursively, so as to express the Gottlieb groups of $\map(X,Y)$ directly in terms of the Gottlieb groups of $Y$.  This yields the formulas of \thmref{introthm: Gottlieb free loop} and \thmref{thm: intro Gottlieb iterated bouquet} from \secref{sec:intro}, which are included in the next result, as well as their more general versions.  

For $m \geq 1$, denote by $\Lambda _m Y$ the \emph{free $m$-bouquet space of $Y$}, i.e., the space $\map(S_m,Y)$, where $S_m = S^1 \vee \cdots \vee S^1$ is a bouquet of $m$ circles. 
Thus $\Lambda_1Y = \Lambda Y$, and the free $2$-bouquet space of $Y$, $\Lambda_2Y$,  might  also be called the space of ``free figure-eights" in $Y$.
 Iterating this construction, whereby we set $\Lambda _m^1 Y =  \Lambda _m Y$, and $\Lambda _m^N Y = \Lambda _m(\Lambda _m^{N-1} Y)$ for $N \geq 2$, we obtain  $\Lambda _m^N Y$, the \emph{$N$-fold iterated free $m$-bouquet space of $Y$}.  In the special case in which $m = 1$, we have $\Lambda_1^N Y = \Lambda ^N Y$, the $N$-fold iterated free loop space of $Y$.

\begin{corollary}[to \thmref{thm: General Gottlieb group}]%
\label{cor: Gottlieb iterated bouquet}
For $\Lambda _m^N Y$, the $N$-fold iterated free $m$-bouquet space of $Y$, we have
$$G_n(\Lambda _m^N Y) \cong \bigoplus_{j = 0}^{N} m^j {N\choose j} G_{n+j}(Y),$$
for $n \geq 1$.  If $m=1$, then we obtain \thmref{thm: intro Gottlieb iterated bouquet} of the Introduction.
\end{corollary}

\begin{proof}
We work inductively over $N$.  For $N= 1$, \thmref{thm: General Gottlieb group} gives
$$G_n(\Lambda_m Y) = G_n\big(\map(S_m,Y) \big) \cong G_n(Y) \oplus m \,G_{n+1}(Y).$$
Now suppose the decomposition holds for $N \leq k-1$, some $k \geq 2$.  Then 
$$\Lambda_m^k Y = \Lambda_m( \Lambda_m^{k-1} Y) = \map(S_m, \Lambda_m^{k-1} Y),$$
and so \thmref{thm: General Gottlieb group} gives
$$G_n(\Lambda_m^k Y) \cong G_n(\Lambda_m^{k-1} Y)  \oplus m \,G_{n+1}(\Lambda_m^{k-1} Y).$$
By induction, we have that
$$
\begin{aligned}
G_n(\Lambda_m^k Y) & \cong \bigoplus_{j = 0}^{k-1} m^j {k-1\choose j} G_{n+j}(Y)  \oplus m\,\bigoplus_{j = 0}^{k-1} m^j {k-1\choose j} G_{n+j+1}(Y)\\
& = G_n(Y) \oplus \bigoplus_{j = 1}^{k-1} m^j {k-1\choose j} G_{n+j}(Y) \\
&  \ \ \ \ \ \ \ \ \ \ \ \ \ \ \ \ \ \ \oplus  \bigoplus_{j = 0}^{k-2} m^{j+1} {k-1\choose j} G_{n+j+1}(Y) \oplus m^k\, G_{n+k}(Y)
\end{aligned}
$$
$$= G_n(Y) \oplus \left(\bigoplus_{j = 1}^{k-1} m^j \left( {k-1\choose j} + {k-1\choose j-1} \right)G_{n+j}(Y) \right) \oplus m^k\, G_{n+k}(Y).$$
Now 
$${k-1\choose j} + {k-1\choose j-1} = {k \choose j},$$
and the result follows.
\end{proof}

We may replace the bouquet of circles by a more general bouquet of spheres (of different dimensions), and obtain explicit decompositions of the Gottlieb groups of the corresponding function spaces directly in terms of the Gottlieb groups of $Y$, although the expressions become more cumbersome to present.   We offer one example to illustrate the general idea, and then give a summary result.  Denote by $\Lambda _{(2,2, 2)}Y$ the space $\map(S^2 \vee S^2 \vee S^2, Y)$, and by $\Lambda _{(2,2,2)}^N Y$ the iterated version of the same, whereby we set $\Lambda _{(2,2,2)}^1 Y =  \Lambda _{(2,2,2)} Y$, and $\Lambda _{(2,2,2)}^N Y = \Lambda _{(2,2,2)}(\Lambda _{(2,2,2)}^{N-1} Y)$ for $N \geq 2$.

\begin{example}
With the above notation, for $n \geq 1$ we have
$$G_n(\Lambda _{(2,2,2)}^N Y) \cong \bigoplus_{j = 0}^{N} 3^j {N\choose j} G_{n+2j}(Y).$$
\end{example}

We summarize these explicit decompositions in the following result.  Suppose $X$ is an $N$-fold product $X = X_1 \times \cdots \times X_N$, each factor of which suspends to some bouquet of spheres.  For each $i = 1, \dots, N$, we write
$$\Sigma X_i =  S^{(i, 1)+1} \vee \cdots \vee S^{(i, n_i)+1}.$$
Here, a superscript $(i, j)+1$ denotes the dimension of the $j$th sphere in $\Sigma X_i$. 
For example, we could just take $X_i =  S^{(i, 1)} \vee \cdots \vee S^{(i, n_i)}$, a plain bouquet  of spheres.

\begin{theorem}\label{thm: general decomposition}
With $X$ as above, we have isomorphisms
$$G_n\big( \map(X,Y;0) \big) \cong G_n(Y) \oplus \bigoplus_{r=1, \dots, N} G_{ n + (i_1, m_1) + \cdots + (i_r,m_r)} (Y),$$
where the sum, for each $r$, runs over all possible $r$-tuples  $1 \leq i_1 < i_2 < \cdots < i_r \leq N$ and all $m_j$ with $1 \leq m_j \leq n_{i_j}$.
\end{theorem}

\begin{proof}
We argue by induction over $N$, the number of factors in $X$.  For $N=1$, we have
$$G_n\big( \map(X,Y;0) \big) \cong G_n(Y) \oplus \bigoplus_{1\leq m \leq n_1} G_{n+(1,m)}(Y),$$
from \thmref{thm: General Gottlieb group}.

Now assume the statement for a product of $N-1$ factors, and write $X = X_1 \times X_{(2, \dots, N)}$, where $X_{(2, \dots, N)} = X_2 \times \cdots \times X_N$.  Then we have
$$\map(X,Y;0) = \map\big( X_1, \map( X_{(2, \dots, N)}, Y) \big),$$
and hence 
$$G_n\big( \map(X,Y;0) \big) \cong G_n\big(\map( X_{(2, \dots, N)}, Y)\big) \oplus \bigoplus_{1\leq m \leq n_1} G_{n+(1,m)}\big(\map( X_{(2, \dots, N)}, Y)\big),$$
again from \thmref{thm: General Gottlieb group}.
Applying the induction hypothesis to each of these terms completes the induction step, and the result follows.
\end{proof}

The decomposition of \thmref{thm: general decomposition} leads to the following observation, which we develop at the end of the paper.

\begin{remark}
Let $X$ be a space that satisfies the hypotheses of \thmref{thm: general decomposition}.  Then, for each $n \geq 1$, $G_n(\map(X,Y;0))$ may be written as a direct sum of $G_n(Y)$ and Gottlieb groups of $Y$ of dimension greater than $n$. 
\end{remark}

In case $X$ is a suspension, or product of such, that does not necessarily split after a further suspension, we may also obtain fairly explicit decompositions for the  (generalized) Gottlieb groups of $\map(X, Y)$
in terms of  (generalized) Gottlieb groups of $Y$.  
From  \corref{cor: Gottlieb function space}, we obtain the following decomposition.  Recall, from our hypotheses established in \secref{sec: Gottlieb function space},  that the spaces that appear suspended in the next two results, namely $A$, $B$, or $B_i$, need not be connected.

\begin{corollary}\label{cor: generalized Gottlieb suspension} 
For  $n \geq 1$, we have isomorphisms of abelian groups 
$$G_n\big(
\map(\Sigma B, Y)\big) \cong G_n(Y) \oplus \mathcal{G}(\Sigma^{n+1} B, Y)$$
and
$$\mathcal{G}\big(\Sigma A , \map(\Sigma B, Y)\big) \cong \mathcal{G}(\Sigma A, Y) \oplus \mathcal{G}(\Sigma^2(A \wedge B), Y). \qed$$
\end{corollary}
\noindent{}Then this result may be used iteratively, to obtain decompositions for $X$ a product of suspensions.  We state one result for the case in which $k = 2$, to give the general idea.  Here, the combinatorial aspects of  the decompositions are not so elegant as above.

\begin{corollary}
For  $X = \Sigma B_1 \times \Sigma B_2$ we have
$$\begin{aligned}
G_n\big( \map(X, Y)\big) \cong & \\
G_n(Y)\, \oplus \,& \mathcal{G}(\Sigma^{n+1} B_1, Y) \oplus \mathcal{G}(\Sigma^{n+1} B_2, Y)
\oplus \mathcal{G}(\Sigma^{n+2}( B_1 \wedge B_2), Y),\\
\end{aligned}$$
\end{corollary}  

\noindent{}Note that there are various ways of re-writing terms of the form $\Sigma^{n+2}(B_1 \wedge B_2)$, which could be helpful in further decomposing the last term in specific cases.

\subsection{Fox Torus Groups and Fox-Gottlieb Groups}  
 
We discuss a relation between our  results above, and results on the torus homotopy groups, originally studied by Fox \cite{Fox}.  These non-abelian groups are defined as 
$$\tau_n(Y) = \pi_1\big(  \map(T^{n-1}, Y; 0) \big),$$
for each $n \geq 2$.  Equivalently,  as in \cite{G-G-W}, they may be defined as
$$\tau_n(Y) = [\Sigma(T^{n-1} \sqcup \{*\}), Y],$$
the set of based homotopy equivalence classes of maps from the (reduced) suspension $\Sigma(T^{n-1} \sqcup \{*\})$, where $T^{n-1} \sqcup \{*\}$ denotes the $(n-1)$-torus with a disjoint base point.  For $n = 1$, then, we take $\tau_1(Y) = \pi_1(Y)$, the ordinary fundamental group of $Y$.

It seems that various relations amongst these groups, and relations between results about these groups and our results above,  may be elucidated by consideration of free loop fibrations.   The basic connection is that, for $n \geq 2$, we may rewrite $\map(T^{n-1}, Y; 0)$ as the iterated free loop space $\Lambda^{n-1}Y$, and thus we have $\tau_n(Y) = \pi_1(\Lambda^{n-1}Y)$.

First, from a free loop fibration $\Omega Z \to \Lambda Z \to Z$, which has a section, we obtain a split short exact sequence of homotopy groups
$$\xymatrix{0 \ar[r]& \pi_i(\Omega Z) \ar[r] & \pi_i(\Lambda Z) \ar[r] & 
\pi_i(Z) \ar[r] \ar@/_1pc/[l] & 1},$$
for each $i \geq 1$.  This leads to direct sums
$$\pi_i(\Lambda Z) \cong  \pi_{i+1}(Z) \oplus \pi_i(Z),$$
for the higher homotopy groups $i \geq 2$ of the free loop space $\Lambda Z$, and a semi-direct product
$$\pi_1(\Lambda Z) \cong \pi_2(Z)  \rtimes \pi_1(Z)$$
for the fundamental group of $\Lambda Z$.  Now by applying these expressions iteratively, beginning with the free loop fibration  
$$\Omega\Lambda^{N-1} Y \to \Lambda^{N} Y \to \Lambda^{N-1} Y,$$
we obtain a direct sum formula 
$$\pi_i(\Lambda^N Y) \cong  \bigoplus_{r = 0}^{N} {N\choose r} \pi_{i + r}(Y),$$
for the higher homotopy groups $i \geq 2$ of the iterated free loop space, in terms of the homotopy groups of $Y$.   For the fundamental group, we have a sequence of split extensions
\begin{equation}\label{eq: extension iterated free loop}
\xymatrix{0 \ar[r]& \pi_1(\Omega\Lambda^{N-1} Y) \ar[r] & \pi_1(\Lambda^{N} Y) \ar[r] & 
\pi_1(\Lambda^{N-1} Y) \ar[r] \ar@/_1pc/[l] & 1},
\end{equation}
for $N \geq 1$, with $\Lambda^0Y$ meaning $Y$. 

Now we observe the following identity.

\begin{lemma}
For any $Z$, we have $\Omega(\Lambda Z) = \Lambda (\Omega Z)$. 
\end{lemma}

\begin{proof}
Indeed, either may be identified by adjunction with the subspace of $\map(S^1 \times S^1, Z)$ that consists of maps $F$ with $F(s_0, s) = z_0$. 
\end{proof}

\noindent{}From this identity, we may write the term $\pi_1(\Omega\Lambda^{N-1} Y)$ that appears in (\ref{eq: extension iterated free loop}) as $\pi_1\big(\Lambda^{N-1} (\Omega Y)\big)$, which results in the following split short exact sequences of Fox torus homotopy groups
$$\xymatrix{0 \ar[r]& \tau_N(\Omega Y) \ar[r] & \tau_{N+1}(Y) \ar[r] & 
\tau_{N}(Y)  \ar[r] \ar@/_1pc/[l] & 1}.$$
Alternatively, we may write the term $\pi_1(\Omega\Lambda^{N-1} Y)$ from (\ref{eq: extension iterated free loop}) as $\pi_2(\Lambda^{N-1} Y)$, and then by using the above direct sum expression, we may write (\ref{eq: extension iterated free loop}) as
$$\xymatrix{0 \ar[r]& \bigoplus_{r = 0}^{N-1} {N\choose r} \pi_{2 + r}(Y) \ar[r] & \tau_{N+1}(Y) \ar[r] & 
\tau_{N}(Y)  \ar[r] \ar@/_1pc/[l] & 1}.$$
These expressions retrieve the result of Fox, which appears in Theorems 1.2 and 1.3 in \cite{G-G-W}. 

In \cite[Def.2.1]{G-G-W}, the $i$th \emph{Fox-Gottlieb group} is defined as the image of the homomorphism induced by the evaluation map $\omega\colon \aut Y \to Y$ on $i$th Fox torus groups.  This is a homomorphism
$$\tau_n(\omega) \colon  \pi_1\big(  \map(T^{n-1}, \aut Y) \big) \to \pi_1\big(  \map(T^{n-1}, Y; 0) \big),$$
whose image is denoted by $G\tau_n(Y) \subseteq \tau_n(Y)$.  When $n = 1$, we retrieve $G\tau_1(Y) = G_1(Y)$, the ordinary first Gottlieb group of $Y$.   When $n \geq 2$, we may identify the  homomorphism induced  by the evaluation map on $n$th Fox torus groups as 
$$\tau_n(\omega) = (\omega_*)_\# \colon   \pi_1\big(  \Lambda ^{n-1} \aut Y \big) \to \pi_1\big( \Lambda ^{n-1} Y \big),$$
and thus we have 
\begin{equation}\label{eq:Fox Gottlieb}
G\tau_n(Y) = G_1(\Lambda ^{n-1} Y)   \cong \bigoplus_{j = 0}^{n-1}  {n-1\choose j} G_{n-1+j}(Y)
\end{equation}
which, omitting  the middle term,  retrieves Theorem 2.1 of  \cite{G-G-W}.  Other identities from \cite{G-G-W} may also be deduced from properties of the free loop fibration.    Because we may identify $\map(T^n, Y;0) = \map(T^{n-1}, \Lambda Y; 0)$, it follows from  the definition that we have
$$\tau_n(Y) = \tau_{n-1}(\Lambda Y) = \cdots = \pi_1(\Lambda ^{n-1} Y).$$
Likewise, we have equalities of Fox-Gottlieb groups
$$G\tau_n(Y) = G\tau_{n-1}(\Lambda Y) = \cdots = G_1(\Lambda ^{n-1} Y).$$
Hence, the identities (\ref{eq:Fox Gottlieb})  for Fox-Gottlieb groups in terms of Gottlieb groups may be used to write, first of all,
 $$G\tau_n(Y) =  \bigoplus_{j = 0}^{n-1}  {n-1\choose j} G_{n-1+j}(Y),$$
and then, via $G\tau_n(Y) = G\tau_{n-1}(\Lambda Y)$, as 
 $$G\tau_n(Y) =  \bigoplus_{j = 0}^{n-2}  {n-2\choose j} G_{n-2+j}(\Lambda Y)$$
Equating the two right-hand sides here gives a relation between the ordinary Gottlieb groups of $Y$ and of $\Lambda Y$ that may be used recursively to retrieve our formulas for the Gottlieb groups of iterated free loop spaces.

\section{Relative Free Loop  Spaces}\label{sec:Pullbacks}

The preceding results that concern free loop and free bouquet spaces may be relativized.  For the free loop space, this is  as follows.  Let $f \colon X \to Y$ be a based map.  The pullback of the fibration $P \colon PY \to Y \times Y$, where $P(\alpha) = \big( \alpha(0), \alpha(1) \big)$, along the map $(f,f) = \Delta_Y\circ f = (f\times f)\circ\Delta_X \colon X \to Y \times Y$ results in a fibration with fibre $\Omega Y$, thus:
$$\xymatrix{\Omega Y \ar@{=}[r] \ar[d] & \Omega Y \ar[d]\\
L_f Y \ar[r] \ar[d]_{\ev_f} & PY \ar[d]^{P}\\
X \ar[r]_-{(f,f)} & Y \times Y.}$$
In case we have $f = \mathrm{id}_Y\colon Y \to Y$, then $L_fY$  reduces to the free loop space $\Lambda  Y$.  We discuss this space motivated in part by its appearance as an object of interest in string topology (cf.~\cite{G-S08, Nai11}).

The fibration $\ev_f\colon L_fY \to X$ admits a standard null-section, which we denote by $\sigma_f\colon X \to L_fY$.  It may be identified as the whisker map induced in the following pullback diagram
$$\xymatrix{  X \ar@/^1pc/[rrd]^{C_f} \ar@/_1pc/[ddr]_{1_X} \ar@{..>}[rd]^{\sigma_f} \\
 & L_f Y \ar[r] \ar[d]^{\ev_f} & P Y \ar[d]^{P}\\
 & X \ar[r]_-{(f,f)} & Y \times Y,}$$
in which $C_f\colon X \to PY$ denotes the map $C_f(x) := C_{f(x)}$, the constant path in $Y$ at $f(x)$. 

We may equally well regard this \emph{relative free loop space} $L_fY$ as being obtained as a pullback of the free loop fibration $\ev\colon \Lambda  Y \to Y$ along $f\colon X \to Y$.  Indeed,  the above pullback may be factored as a composition of pullbacks as follows:
$$\xymatrix{
L_f Y \ar[r] \ar[d] & \Lambda  Y \ar[d]^{\ev} \ar[r] & PY \ar[d]^{P}\\
X \ar[r]_-{f} & Y \ar[r]_-{\Delta_Y} & Y \times Y.}$$
Then we obtain a canonical map $(\ev, f_*)\colon \Lambda  X \to L_fY$, determined by $f\colon X \to Y$, as the whisker map induced in the following diagram:
$$\xymatrix{ \Lambda  X \ar@/^1pc/[rrd]^{f_*} \ar@/_1pc/[ddr]_{\ev} \ar@{..>}[rd] \\
 & L_f Y \ar[r] \ar[d] & \Lambda  Y \ar[d]^{\ev}\\
 & X \ar[r]_-{f} & Y .}$$

Recall from the Introduction the relative version of the Gottlieb groups, namely 
$$G_n(Y,X;f) =  \mathrm{im}\{ \omega_{\#}\colon \pi_n\big(\map(X,Y;f)\big) \to \pi_n(Y)\},$$
for each $n\geq 1$, where $\omega\colon \map(X,Y;f) \to Y$ evaluates a map at the base point. 
Our results here will be expressed in terms of these relative Gottlieb groups.

Our basic result  is a direct generalization of \thmref{introthm: Gottlieb free loop} of the Introduction.
\begin{theorem}\label{thm: relative Gottlieb free loop}
For $n \geq 1$, we have a split extension
$$
\xymatrix{0 \ar[r]& G_{n+1}(Y,X;f) \ar[r] & G_n\big(L_fY, \Lambda  X; (\ev, f_*)\big)  \ar[r] & 
G_n(X) \ar[r] \ar@/_1pc/[l] & 1}.
$$
For each $n \geq 2$, this is a split short exact sequence of abelian groups, and we have an isomorphism of abelian groups
$$G_n\big(L_fY, \Lambda  X; (\ev, f_*)\big) \cong G_n(X) \oplus G_{n+1}(Y,X;f).$$
\end{theorem}

\begin{proof}
The map $f\colon X \to Y$ induces $f_*\colon \aut X \to \map(X,Y;f)$, which we may use to construct $L_{f_*}\map(X,Y;f)$ as the pullback of $P\colon P\map(X,Y;f) \to \map(X,Y;f) \times \map(X,Y;f)$ along  
$\Delta \circ f_*\colon \aut  X \to \map(X,Y;f) \times \map(X,Y;f)$.  The resulting fibre sequence, displayed as the left-hand column in the following pullback diagram   
$$\xymatrix{\Omega \map(X,Y;f) \ar@{=}[r] \ar[d] & \Omega \map(X,Y;f) \ar[d] \\
L_{f_*}\map(X,Y;f) \ar[r] \ar[d] & P\map(X,Y;f) \ar[d]^{P} \\
  \aut  X  \ar[r]_-{\Delta \circ f_*} & \map(X,Y;f) \times \map(X,Y;f), }$$
  fits into a ladder of fibre sequences, displayed as the rows in the following diagram
$$\xymatrix{
 & &  \map\big(\Lambda  X, L_fY; (\ev, f_*) \big) \ar@<1ex>[d]^{r} \\
 \Omega \map(X,Y;f) \ar[d]_{\Omega_\ev} \ar[rr]^-{j_{f_*}}  & & L_{f_*}\map(X,Y;f)\ar[r]_-{\ev_{f_*}} 
\ar[d]_{(\omega_B)_{*}}  \ar@<1ex>[u]^{\Phi}
   & \aut X  \ar[d]^{\omega_X} \ar@/_1pc/[l]_-{\sigma_{f_*}} \\
\Omega Y \ar[rr]_-{j_f}  &  & L_fY \ar[r]^-{\ev_f}    & X \ar@/^1pc/[l]^-{\sigma_f}.  } $$
We describe the retraction $L_{f_*}\map(X,Y;f) \to \map\big(\Lambda  X, L_fY; (\ev, f_*)  \to L_{f_*}\map(X,Y;f)$.  As a first step, we identify 
$$L_{f_*}\map(X,Y;f) \equiv \map( X, L_fY; \sigma_f).$$
This follows from an identification of pullback squares
\begin{displaymath}
\xymatrix@C=-32pt{ &
\map( X, L_fY; \sigma_f)
  \ar[rr]
\ar'[d]^-{(\ev_f)_*}[dd] & &
\map(X, PY)
\ar[dd]^-{P_*} \\
L_{f_*}\map(X,Y;f)
\ar[ru] \ar[rr]
  \ar[dd]_{\omega_\#}& &
P\, \map(X, Y; f)
\ar@{=}[ru]
\ar[dd]^(0.3){\omega_\#}\\
  &  \map(X,X; 1_X)
\ar'[r]_(0.7){(f,f)_*}[rr] & &
\map(X, Y\times Y)  \\
\aut X \ar@{=}[ru]
\ar[rr]_-{(f_*, f_*)}
& & \map(X, Y; f) \times \map(X, Y; f) \ar@{=}[ru] }
\end{displaymath}
Then the argument is a direct generalization of that used in the proof of \thmref{thm: split short Gottlieb}.
Notice that, in degrees $\geq 2$, all groups concerned are abelian: the ordinary Gottlieb groups are always abelian, and  $G_{r}(Y,X;f) \subseteq \pi_r(Y)$, which is abelian for $r \geq 2$.
\end{proof}

We may iterate this construction, and obtain decompositions that are relative versions of those in 
\secref{sec:Free Bouquets}.  We indicate this direction, but do not attempt any great generality. To iterate,   we regard $L_fY$ as the pullback of $\ev\colon \Lambda  Y \to Y$ along $f\colon X \to Y$.  Then the next step is to form $L^2_fY$ as the pullback of $\ev\colon \Lambda  L_fY \to L_fY$ along $(\ev, f_*)\colon \Lambda  X \to L_fY$.  We obtain a map
$(\ev,f_*)^2\colon \Lambda ^2X \to L^2_fY$ as the whisker map in the following pullback diagram:
$$\xymatrix{ \Lambda ^2 X \ar@/^1pc/[rrd]^{\Lambda (\ev, f_*)} \ar@/_1pc/[ddr]_{\ev} \ar@{..>}[rd] \\
 & L^2_f Y \ar[r] \ar[d] & \Lambda  L_fY \ar[d]^{\ev}\\
 & \Lambda  X \ar[r]_-{(\ev, f_*)} & L_fY .}$$
We state, without proof, the decomposition for the abelian relative Gottlieb groups of the iterated construction:

\begin{corollary}
For each $n \geq 2$, we have an isomorphism of abelian groups
$$G_n\big(L^2_fY, \Lambda ^2 X; (\ev, f_*)^2\big) \cong G_n(X) \oplus 2 G_{n+1}(X) \oplus G_{n+2}(Y,X;f). \qed$$
\end{corollary}

If $f\colon X \to Y$ is the identity, then this decomposition and that of the previous result reduce to those in \secref{sec:Free Bouquets}.
We mention also that, as we did with the free loop space, we may extend this construction to a free bouquet space, and also iterate it, so as to obtain what we could call \emph{iterated, relative free bouquet spaces}.  In the first instance, the effect would be to introduce a coefficient of $m$ before the 
$G_{n+1}(Y,X;f)$ term in 
\thmref{thm: relative Gottlieb free loop}.

\section{Consequences and Concluding Remarks}\label{sec:String}

It is well-known that, if  $Y$ is an $H$-space, then we have $G_*(Y) = \pi_*(Y)$.   A space that satisfies $G_*(Y) = \pi_*(Y)$ is known as a \emph{$G$-space}. 
Work of Ganea and others has produced separating examples that illustrate, in general, one may have a  $G$-space that is not an $H$-space.  In \cite{Agu87}, Aguad{\'e} introduced the notion of a \emph{$T$-space}, namely a space for which the free loop fibration is fibre-homotopically trivial.  Again, an $H$-space is a $T$-space, but a separating example that illustrates that, generally, a $T$-space need not be an $H$-space is given in \cite{Agu87}.  Also, any $T$-space must be a $G$-space (\cite[Th.2.2]{W-Y95}---a stronger statement holds: see in the proof below).  We are unaware of a separating example between $T$-spaces and $G$-spaces.
\thmref{thm: General Gottlieb group} leads to the following observations in this area.

\begin{corollary}\label{cor: G-space}
\begin{enumerate}
\item If $Y$ is a $T$-space, then so is $\map(X,Y;0)$.

\item Suppose that $\Sigma X$ splits as some wedge of spheres.  If $Y$ is a $G$-space, then so is $\map(X,Y;0)$.
\end{enumerate}
\end{corollary}

\begin{proof}
(1) In \cite[Th.2.2]{W-Y95}, it is shown that $Y$ is a $T$-space if, and only if, we have $\mathcal{G}(\Sigma B, Y) = [\Sigma B, Y]$, for any suspension $\Sigma B$.  Consider the (split) short exact sequence of image subgroups, included into the (split) short exact sequence of homotopy sets, from which we obtained 
\thmref{thm: split short Gottlieb}.  Namely, from the proof of that result, we have the following inclusion of one short exact sequence in another:
$$\xymatrix{ 0 \ar[r] & \mathcal{G}(\Sigma(A\wedge X), Y) \ar[r] \ar@{^{(}->}[d] &
\mathcal{G}\big(\Sigma A, \map(X,Y;0) \big) \ar[r] \ar@{^{(}->}[d]& \mathcal{G}(\Sigma A, Y)  \ar[r] \ar@{^{(}->}[d]& 0\\
0 \ar[r] & [\Sigma(A\wedge X), Y] \ar[r] &
[\Sigma A, \map(X,Y;0) ] \ar[r] & [\Sigma A, Y]  \ar[r] & 0}$$
Since $Y$ is a $T$-space, the left and right inclusions are equalities, which implies the middle inclusion is an equality.  From the criterion of \cite[Th.2.2]{W-Y95}, we obtain that $\map(X,Y;0)$ is a $T$-space.

(2) Now suppose that $Y$ is a $G$-space.  We take $A = S^{n-1}$, for $n \geq 1$, in the above diagram and obtain the following inclusion of one short exact sequence in another:
$$\xymatrix{ 0 \ar[r] & \mathcal{G}(\Sigma^n X, Y) \ar[r] \ar@{^{(}->}[d] &
G_n\big( \map(X,Y;0) \big) \ar[r] \ar@{^{(}->}[d]& G_n(Y)  \ar[r] \ar@{^{(}->}[d]& 0\\
0 \ar[r] & [\Sigma^n X, Y] \ar[r] &
\pi_n\big( \map(X,Y;0) \big) \ar[r] & \pi_n(Y)  \ar[r] & 0}$$
Our assumption on $\Sigma X$ implies that the left-hand inclusion decomposes into a product of inclusions of Gottlieb groups of $Y$ into the corresponding homotopy groups of $Y$.  Since $Y$ is a $G$-space, left and right inclusions are equalities, and it follows  that the middle inclusion is an equality.  That is, $\map(X,Y;0)$ is a $G$-space.  
\end{proof}

\begin{remark}
It is easily seen that the property of being a $G$-space, respectively a $T$-space, is inherited by retracts (cf.~\cite[Cor.2.10]{W-Y95} for the latter).  Since $\map(X,Y;0)$ retracts onto $Y$, the statements in \corref{cor: G-space} may be replaced by two-way implications.  Notice that the first includes  \cite[Th.2.12]{W-Y95}.  But the interest in \corref{cor: G-space} comes from the fact that we are able to deduce properties of the function space $\map(X,Y)$ from properties of $Y$ alone, rather than the other way around.  This is the point of view developed in our remaining observations.
\end{remark}

As we have seen in \thmref{thm: General Gottlieb group}, if  $\Sigma X$ splits as a bouquet of spheres, then \corref{cor: Gottlieb function space} yields a decomposition of $G_*(\map(X,Y;0))$ in terms of the Gottlieb groups of $Y$.  

\begin{example}
Suppose we take $X = S^5 \cup_\alpha e^{10}$, with $\alpha \in \pi_9(S^5) \cong \mathbb{Z}_2$ the non-zero element.  Since $\pi_{10}(S^6) = 0$, we have that $\Sigma X \simeq S^6 \vee S^{10}$.  From
\corref{cor: Gottlieb function space}, we obtain
$$G_n(\map(S^5 \cup_\alpha e^{10}, Y;0)) \cong G_n(Y) \oplus G_{n+5}(Y) \oplus G_{n+10}(Y),$$
for $n \geq 1$.
\end{example} 

Likewise, if we have a space $X$ for which some suspension $\Sigma^n X$ splits as a wedge of spheres, then all but finitely many of the Gottlieb groups of $\map(X,Y)$ may be expressed as direct sums of the Gottlieb groups of $Y$.

Pursuing this line somewhat leads to a strong consequence for the global structure of Gottlieb groups of  function spaces.   Suppose $Y$ is a simply connected, finite complex.  Then a result of F{\'e}lix-Halperin (cf.~\cite[Prop.28.8]{FHT01}) implies that the Gottlieb groups of $Y$ are finite groups in all but a finite number of degrees.  

The function spaces that we consider here are neither simply connected, nor finite complexes in general.  But from our results above, we deduce the following intriguing fact: if $X$ and $Y$ are finite complexes, and $Y$ is simply connected, then $G_n\big(\map(X,Y;0)\big)$ is a finite group in all but a finite number of degrees.  We first prove \thmref{thm: gamma formula} of the Introduction.  This result explicitly identifies the rational Gottlieb groups of $\map(X,Y)$ in terms of the rational homology of $X$ and the rational Gottlieb groups of $Y$, which are both fairly amenable to calculation.

\begin{proof}[Proof of \thmref{thm: gamma formula}]
\corref{cor: Gottlieb function space} gives  an isomorphism of abelian groups
$$G_n\big(\map(X, Y;0)\big) \cong G_n(Y) \oplus \mathcal{G}(\Sigma^{n} X, Y),$$
for each $n \geq 1$.  Hence $\gamma_n\big(\map(X, Y;0)\big) = \gamma_n(Y) + \mathrm{dim}_\Q\big( \mathcal{G}(\Sigma^{n} X, Y)\otimes\Q\big).$
Here, $\mathrm{dim}_\Q(-)$ refers to the dimension as a rational vector space. We focus on this last term.  Let $l\colon Y \to Y_\Q$ denote the rationalization of $Y$.  The homomorphism induced by $l$ rationalizes the group $[\Sigma^{n} X, Y]$.  That is, we have $l_*\big([\Sigma^{n} X, Y]\big) \cong [\Sigma^{n} X, Y] \otimes \Q$ as abelian groups.  This restricts  to generalized Gottlieb groups, and gives  
$l_*\big(\mathcal{G}(\Sigma^{n} X, Y)\big) \cong \mathcal{G}(\Sigma^{n} X, Y) \otimes \Q$.  Furthermore, because $Y$ is finite, we may identify  
$l_*\big(\mathcal{G}(\Sigma^{n} X, Y)\big) = \mathcal{G}(\Sigma^{n} X, Y_\Q)$,  (see 
\cite[Th.2.2]{Lan75}, and \cite[Th.3.2]{W-K86}).  

Now $\Sigma^n X$ has the same rational homotopy type as some wedge of spheres (see \cite[Th.24.5]{FHT01}).  Since a rational homotopy equivalence preserves Betti numbers, we must have a rational homotopy equivalence
$$\theta \colon \bigvee_{i =1}^{\textrm{dim}\,X}\ \beta_i(X) S^{n+i} \to \Sigma^n X.$$
Here, for $b$ a non-negative integer, $bS^n$ denotes the $n$-fold wedge $S^n \vee \cdots \vee S^n$.
But then the isomorphism $\theta^*\colon [\Sigma^{n} X, Y_\Q] 
\to  [\vee_{i =1}^{\textrm{dim}\,X}\ \beta_i(X) S^{n+i},  Y_\Q]$ restricts to an isomorphism
$\theta^*\colon \mathcal{G}(\Sigma^{n} X, Y_\Q) 
\to  \mathcal{G}(\vee_{i =1}^{\textrm{dim}\,X}\ \beta_i(X) S^{n+i},  Y_\Q)$, and we may write
 $$\mathcal{G}(\bigvee_{i =1}^{\textrm{dim}\,X}\ \beta_i(X) S^{n+i},  Y_\Q) \cong \oplus_{i =1}^{\textrm{dim}\,X}\ \beta_i(X) \mathcal{G}( S^{n+i},  Y_\Q) = \oplus_{i =1}^{\textrm{dim}\,X}\ \beta_i(X) G_{n+i}(Y_\Q).$$
Once again, since $Y$ is finite, we have $G_{n+i}(Y)\otimes\Q \cong G_{n+i}(Y_\Q)$, and so we have
$$\mathrm{dim}_\Q\big( \mathcal{G}(\Sigma^{n} X, Y)\otimes\Q\big) = \sum_{i=1}^{\textrm{dim}\,X}\ \beta_i(X) \gamma_{n+i}(Y).$$
Combining this with the $G_n(Y)$ term gives the result.
\end{proof}

\begin{corollary}\label{cor:finite Gottlieb}
Let $X$ and $Y$ be finite  complexes,  with $Y$ simply connected.    
Then $G_n\big(\map(X,Y;0)\big)$ is a finite group for  all but  finitely many $n$.    If $N$ is the highest degree in which $G_N(Y)$ has positive rank,  then $N$ is also the highest degree in which  $G_N\big(\map(X,Y;0)\big)$ has positive rank, and $\gamma_N\big(\map(X,Y;0)\big) = \gamma_N(Y)$.
\end{corollary}

\begin{proof}
The result of F{\'e}lix-Halperin \cite[Prop.28.8]{FHT01} mentioned above  guarantees there is some  $N$ that is the highest degree in which $G_N(Y)$ has positive rank.  The remaining assertions follow from \thmref{thm: gamma formula}
\end{proof}

Finally, we observe that \thmref{thm: General Gottlieb group} could, in  the right circumstances, provide a necessary condition for a space to be a function space of the kind considered there.  For example, suppose that we have a fibration $\Omega Y \to E \to Y$ that admits a section, for some space $Y$. One might ask whether $E$ is of the homotopy type of  $\Lambda Y$?  \thmref{thm: General Gottlieb group} requires that  $G_*(E) \cong G_*(Y) \oplus G_{*+1}(Y)$ for this to be possible.  



\begin{thebibliography}{10}

\bibitem{Agu87}
J.~Aguad{\'e}, \emph{Decomposable free loop spaces}, Canad. J. Math.
  \textbf{39} (1987), no.~4, 938--955. \MR{915024 (88m:55012)}

\bibitem{FHT01}
Y.~F{\'e}lix, S.~Halperin, and J.-C. Thomas, \emph{Rational homotopy theory},
  Graduate Texts in Mathematics, vol. 205, Springer-Verlag, New York, 2001.
  \MR{1802847 (2002d:55014)}

\bibitem{F-T-V04}
Y.~Felix, J.-C. Thomas, and M.~Vigu{\'e}-Poirrier, \emph{The {H}ochschild
  cohomology of a closed manifold}, Publ. Math. Inst. Hautes \'Etudes Sci.
  (2004), no.~99, 235--252. \MR{2075886 (2005g:57054)}

\bibitem{Fox}
R.~H. Fox, \emph{Homotopy groups and torus homotopy groups}, Ann. of Math. (2)
  \textbf{49} (1948), 471--510. \MR{0027143 (10,260f)}

\bibitem{G-G-W}
M.~Golasi{\'n}ski, D.~Gon{\c{c}}alves, and P.~Wong, \emph{Generalizations of
  {F}ox homotopy groups, {W}hitehead products and {G}ottlieb groups}, Ukra\"\i
  n. Mat. Zh. \textbf{57} (2005), no.~3, 320--328. \MR{2188429 (2006i:55018)}

\bibitem{Go-Mu08}
M.~Golasi{\'n}ski and J.~Mukai, \emph{Gottlieb groups of spheres}, Topology
  \textbf{47} (2008), no.~6, 399--430. \MR{2427733 (2009j:55017)}

\bibitem{Go1}
D.~H. Gottlieb, \emph{Evaluation subgroups of homotopy groups}, Amer. J. Math.
  \textbf{91} (1969), 729--756. \MR{43 \#1181}

\bibitem{G-S08}
K.~Gruher and P.~Salvatore, \emph{Generalized string topology operations},
  Proc. Lond. Math. Soc. (3) \textbf{96} (2008), no.~1, 78--106. \MR{2392316
  (2009e:55013)}

\bibitem{Lan75}
G.~E. Lang, Jr., \emph{Localizations and evaluation subgroups}, Proc. Amer.
  Math. Soc. \textbf{50} (1975), 489--494. \MR{0367986 (51 \#4228)}

\bibitem{McD08}
D.~McDuff, \emph{The symplectomorphism group of a blow up}, Geom. Dedicata
  \textbf{132} (2008), 1--29. \MR{2396906 (2009c:57044)}

\bibitem{Nai11}
T.~Naito, \emph{On the mapping space homotopy groups and the free loop space
  homology groups}, Algebr. Geom. Topol. \textbf{11} (2011), no.~4, 2369--2390.
  \MR{2835233 (2012g:55010)}

\bibitem{Spa81}
E.~H. Spanier, \emph{Algebraic topology}, Springer-Verlag, New York, 1981,
  Corrected reprint. \MR{666554 (83i:55001)}

\bibitem{Ste67}
N.~E. Steenrod, \emph{A convenient category of topological spaces}, Michigan
  Math. J. \textbf{14} (1967), 133--152. \MR{0210075 (35 \#970)}

\bibitem{Var69}
K.~Varadarajan, \emph{Generalised {G}ottlieb groups}, J. Indian Math. Soc.
  (N.S.) \textbf{33} (1969), 141--164 (1970). \MR{0281207 (43 \#6926)}

\bibitem{W-K86}
M.~H. Woo and J.-R. Kim, \emph{Localizations and generalized {G}ottlieb
  subgroups}, J. Korean Math. Soc. \textbf{23} (1986), no.~2, 151--157.
  \MR{888335 (88g:55024)}

\bibitem{W-Y95}
M.~H. Woo and Y.~S. Yoon, \emph{{$T$}-spaces by the {G}ottlieb groups and
  duality}, J. Austral. Math. Soc. Ser. A \textbf{59} (1995), no.~2, 193--203.
  \MR{1346627 (96e:55009)}

\end{thebibliography}
     
\def\cprime{$'$}
\providecommand{\bysame}{\leavevmode\hbox to3em{\hrulefill}\thinspace}
\providecommand{\MR}{\relax\ifhmode\unskip\space\fi MR }
\providecommand{\MRhref}[2]{%
  \href{http://www.ams.org/mathscinet-getitem?mr=#1}{#2}
}
\providecommand{\href}[2]{#2}

\end{document}